\documentclass[12pt]{article}

\usepackage[OT1]{fontenc}

\usepackage{amsmath,amssymb,amsthm}
\usepackage{array,booktabs}
\usepackage{enumitem}
\usepackage{titlesec}
\usepackage[hidelinks]{hyperref}
\hypersetup{
  pdftitle={Diameter and radius of an additive-multiplicative graph},
  pdfauthor={Yi Hu and Nan Yang},
  pdfkeywords={Arithmetic graph, diameter, radius, primes in linear forms,
    almost primes, modular obstruction}
}
\numberwithin{equation}{section}

\titleformat{\section}{\normalfont\Large\bfseries\filcenter}
  {\thesection.}{0.6em}{}
\titleformat{\subsection}{\normalfont\normalsize\itshape}
  {\thesubsection.}{0.6em}{}
\renewenvironment{abstract}
  {\par\small\begin{center}\abstractname\end{center}\noindent\ignorespaces}
  {\par\medskip}

\newtheorem{theorem}{Theorem}[section]
\newtheorem{lemma}[theorem]{Lemma}
\newtheorem{proposition}[theorem]{Proposition}
\newtheorem{corollary}[theorem]{Corollary}
\theoremstyle{definition}
\newtheorem{definition}[theorem]{Definition}
\newtheorem{conjecture}[theorem]{Conjecture}
\newtheorem{algorithm}[theorem]{Algorithm}
\newtheorem{question}[theorem]{Question}
\theoremstyle{remark}
\newtheorem{remark}[theorem]{Remark}

\newcommand{\N}{\mathbb{N}}
\newcommand{\Z}{\mathbb{Z}}
\newcommand{\PP}{\mathcal{P}}
\DeclareMathOperator{\diam}{diam}
\DeclareMathOperator{\rad}{rad}
\DeclareMathOperator{\ecc}{ecc}
\DeclareMathOperator{\lcm}{lcm}
\newcommand{\decjoin}[2]{%
  \left\langle\begin{gathered}\mathtt{#1}\\[-2pt]\mathtt{#2}\end{gathered}\right\rangle_{10}}

\title{Diameter and radius of an additive--multiplicative graph\thanks{%
  The authors proposed the problem and outlined the research directions.
  They relied heavily on GPT-5.6 sol to develop the proofs presented in
  this paper.}}
\author{%
  Yi Hu\thanks{Self-employed. Email:
    \href{mailto:wuhuaguoshuzhi@qq.com}{\nolinkurl{wuhuaguoshuzhi@qq.com}}.}
  \and
  Nan Yang\thanks{Microsoft Research. Email:
    \href{mailto:nanya@microsoft.com}{\nolinkurl{nanya@microsoft.com}}.}}
\date{September 30, 2026}

\begin{document}

\maketitle

\begingroup
\renewcommand{\thefootnote}{}
\begin{NoHyper}
\end{NoHyper}
\endgroup

\begin{abstract}
Let $G$ be the graph on the positive integers in which distinct vertices
are adjacent if they differ by one or their quotient in one order is
prime. We prove that $6\leq\diam(G)\leq7$ and
$4\leq\rad(G)\leq5$.
Under Dickson's conjecture for two linear forms, we prove that
$\diam(G)=6$ and $\rad(G)=4$.
\end{abstract}

\section{Introduction}

Let $\N=\{1,2,\ldots\}$, and let $\PP$ denote the set of primes.
We consider the undirected simple graph $G$ with vertex set $\N$ and
edge relation
\begin{equation}\label{eq:edges}
 x\sim y
 \quad\Longleftrightarrow\quad
 |x-y|=1\quad\text{or}\quad
 \frac{\max\{x,y\}}{\min\{x,y\}}\in\PP.
\end{equation}
Thus a path consists of additions or subtractions of one, multiplications
by a prime, and exact divisions by a prime. Every intermediate value
must be a positive integer.

Write $d(x,y)$ for graph distance, and define
\[
 \ecc(x)=\sup_{y\in\N}d(x,y),\qquad
 \rad(G)=\inf_{x\in\N}\ecc(x),\qquad
 \diam(G)=\sup_{x,y\in\N}d(x,y).
\]
The graph is connected because it contains every edge $n\sim n+1$.
Its multiplicative subgraph has unbounded diameter: for example, the
distance from $1$ to $2^k$ using prime multiplications and divisions
alone is $k$. The addition of the successor edges changes this
behavior: every two vertices of $G$ can be joined in at most seven steps,
and every vertex can be reached from $1$ in at most five steps.
Our main result leaves only two possible values for each of the
diameter and radius.

\begin{theorem}\label{thm:main}
The graph $G$ satisfies
\[
 6\leq\diam(G)\leq7,\qquad 4\leq\rad(G)\leq5.
\]
Moreover, $\ecc(1)\leq5$, and $d(a,b)\leq6$ whenever $a$ and $b$
have the same parity. Under Dickson's conjecture for two linear forms,
one has
\[
 \diam(G)=6,\qquad \rad(G)=\ecc(1)=4.
\]
\end{theorem}

The upper bounds use a theorem of Coleman~\cite{Coleman} on the
equation $xp-ym=h$, where $p$ is prime and $m$ has at most two prime
factors. Such a solution gives a short path between $x$ and $y$.
An auxiliary prime removes the coprimality restrictions on the
endpoints. For the diameter lower bound, we give an explicit pair at
distance six and a finite modular certificate excluding every operation
word of length at most five. The verification does not restrict the
sizes of the prime multipliers or intermediate vertices.
Appendix~\ref{app:exclusion} gives pseudocode for generating the
certificate, a separate checker using backward propagation, and proofs
of their soundness.

The radius lower bound is proved in Section~\ref{sec:radius} by a
counting argument: for each fixed vertex, the vertices of the same
parity within three steps form a set of density zero. Thus every
vertex has eccentricity at least four. The stronger upper bound
$\ecc(1)\leq5$ supplies a vertex witnessing $\rad(G)\leq5$.

Section~\ref{sec:dickson} establishes the conditional equalities in
Theorem~\ref{thm:main}. Dickson's conjecture therefore predicts
diameter six and radius four. Either diameter seven or radius five
would contradict that conjecture; the problem is to establish the
predicted values unconditionally.

Section~\ref{sec:progressions} shows that every arithmetic progression
contains infinitely many vertices within six steps of any fixed vertex.
Consequently, congruence conditions on one endpoint cannot by
themselves force its distance from a fixed vertex to be at least seven.
This identifies a limitation of the method used for the lower bound.

\subsection{Related work}
The multiplicative subgraph of $G$, obtained by deleting the successor
edges, is the cover graph of the divisibility order. Its distance is
\[
 d_{\times}(a,b)=\sum_{p\in\PP}|v_p(a)-v_p(b)|.
\]
This metric appears in Mathar's OEIS entry A127185~\cite{OEISmetric}
and is studied by Dominici~\cite{Dominici}, who also determines the
diameter of finite initial intervals in this metric. On all of $\N$,
the multiplicative distance is unbounded. Our results describe the
uniform bounds that become possible after successor edges are added.

Sabellek's number-theoretic game~\cite{Sabellek} uses the directed
moves $n\mapsto n+1$ and $n\mapsto n/p$, with $p\mid n$ prime,
from positions $n>1$. The original move graph is directed; forgetting
its orientations gives the edge relation~\eqref{eq:edges}. Sabellek
asks about winning positions, and Ryde's OEIS entry
A363369~\cite{OEISdirected} records the minimum number of directed
moves needed to reach $1$. These directed distances differ from the
undirected distances studied here. For example, $d(32,1)=2$, via
$32,31,1$, whereas the directed distance is $3$, via $32,33,11,1$.
Thus the earlier game provides an antecedent for the arithmetic moves,
but its game-theoretic results and directed distances do not establish
the diameter and radius bounds in Theorem~\ref{thm:main}.

Other related graphs use different vertex sets or operations.
Perucca, Seur\'e, and Wolff~\cite{Perucca} study directed graphs on
residue classes modulo $m$, generated by addition of one and
multiplication by a fixed integer, and their associated simple
undirected graphs. K\'atai~\cite{Katai} and subsequent
authors~\cite{PollackShapiroSparer,Collison} study directed graphs on
prime vertices, with an arc $p\to q$ when $q\mid p+c$ for a fixed
positive shift $c$.

The number-theoretic input to our upper bounds belongs to the
Chen-sieve tradition of finding a prime and an almost prime in related
linear forms; see~\cite[Chapter~11]{HalberstamRichert}.
Here $P_r$ denotes a positive integer with at most $r$ prime factors,
counted with multiplicity. Liu~\cite{Liu} treated the equation
$xp-yP_3=h$ for coprime positive $x,y$ and $h\in\{1,2\}$ satisfying
$x+y\equiv h\pmod2$. As Coleman explicitly records
in~\cite[Note (iii), p.~3]{Coleman}, Xie~\cite{Xie} improved $P_3$
to $P_2$ in these cases. We use Coleman's more general theorem in
the form stated as Theorem~\ref{thm:coleman}; its proof also yields
infinitely many solutions. The short paths in
Lemma~\ref{lem:coprime} use this established existence result, and
Lemma~\ref{lem:auxiliary} supplies an auxiliary prime to handle
endpoints that need not be coprime.

To the best of our knowledge, the uniform undirected diameter and
radius bounds in Theorem~\ref{thm:main}, including the conditional
equalities under Dickson's conjecture, have not previously been
established for $G$. The graph-specific arguments include the
distance-six certificate, the density estimate for three-step
neighborhoods, and the arithmetic-progression result of
Section~\ref{sec:progressions}.

\subsection{Notation}
For $n\geq1$, let $\Omega(n)$ be the number of prime factors of $n$,
counted with multiplicity, with $\Omega(1)=0$. Write
$B_k(a)=\{n\in\N:d(a,n)\leq k\}$. Whenever an expression is
required to be prime, it is understood to be a positive integer.
Signs in separate occurrences of $\pm$ are independent unless stated
otherwise.

\section{The upper bound}\label{sec:upper}

We use the following form of Coleman's theorem.

\begin{theorem}[Coleman~{\cite[Theorem 1 and Note (i), pp.~2--3]{Coleman}}]\label{thm:coleman}
Let $b_1,b_2,b_3$ be pairwise coprime positive integers, and suppose
that $2\mid b_1b_2b_3$. There are infinitely many solutions of
\[
 b_1p-b_2m=b_3
\]
in which $p$ is prime and $\Omega(m)\leq2$.
\end{theorem}

\begin{lemma}\label{lem:coprime}
Let $x,y\in\N$ be coprime. If $x$ and $y$ have opposite parity,
then $d(x,y)\leq4$. If both are odd, then $d(x,y)\leq5$.
\end{lemma}

\begin{proof}
In the first case, Theorem~\ref{thm:coleman} gives $xp-ym=1$, with
$p$ prime and $\Omega(m)\leq2$. Multiply $x$ by $p$, subtract one,
and divide by the prime factors of $m$. This gives a path of length
$2+\Omega(m)\leq4$.

In the second case, the coefficients $x,y,2$ are pairwise coprime.
Apply the theorem to $xp-ym=2$ and use two subtractions instead of
one. The path has length at most five. All its vertices are positive.
\end{proof}

\begin{lemma}\label{lem:auxiliary}
Let $a,b\in\N$, and suppose that $a$ is even or $b$ is odd.
There are infinitely many odd primes $p$ such that
\[
 \gcd(ap+1,b)=1.
\]
\end{lemma}

\begin{proof}
For each odd prime $\ell\mid b$, choose a nonzero residue $u_\ell$
modulo $\ell$ such that $au_\ell+1\not\equiv0\pmod\ell$.
If $\ell\mid a$, any nonzero residue works. Otherwise at most one
nonzero residue is forbidden, so a choice exists even when $\ell=3$.
Require also $p\equiv1\pmod2$. The Chinese remainder theorem gives
a reduced residue class satisfying these conditions, and Dirichlet's
theorem gives infinitely many primes in that class.

If $b$ is even, the hypothesis makes $a$ even, so $ap+1$ is odd.
If $b$ is odd, there is no condition at $2$. Thus the chosen primes
satisfy the asserted coprimality condition.
\end{proof}

\begin{proposition}\label{prop:upper}
If $a\equiv b\pmod2$, then $d(a,b)\leq6$. For arbitrary
$a,b\in\N$, one has $d(a,b)\leq7$.
\end{proposition}

\begin{proof}
Suppose first that the endpoints have the same parity. Choose an odd
prime $p$ as in Lemma~\ref{lem:auxiliary} and set $c=ap+1$.
Then $c$ and $b$ are coprime and of opposite parity. Since
$a\sim ap\sim c$, Lemma~\ref{lem:coprime} gives
$d(a,b)\leq2+4=6$.

If the endpoints have opposite parity, exchange them if necessary so
that $a$ is even and $b$ is odd. The same construction makes $c$ and
$b$ coprime odd integers. Hence $d(a,b)\leq2+5=7$.
\end{proof}

\section{An explicit pair at distance six}\label{sec:lower}

The endpoints below were designed by prescribing prime-power
congruences that obstruct short operation words and combining them
by the Chinese remainder theorem. The guiding observation is that
multiplication or exact division by a prime changes any fixed prime
valuation by at most one, whereas adding or subtracting one can
destroy divisibility by an arbitrarily high prime power. Local
conditions exploit this contrast to prevent short paths from matching
the divisibility required at the two endpoints. The size of the
resulting integers reflects the simultaneous congruence conditions.

We verify the resulting pair independently of its construction, using
exact arithmetic and a finite modular certificate. For every operation
word of length at most five, the certificate records a failed modular
test and the branches needed to handle exceptional prime values. A
separate checker verifies these records using backward propagation.
The procedures and their complete input data appear in
Appendix~\ref{app:exclusion}.

For readability, angle brackets with subscript $10$ containing two
lines of digits denote the positive integer obtained by concatenating
the two lines, in their displayed order. Define
\begin{align}
 a&=\decjoin{787713964471027128589830600283640405882003639658}
                  {00348300294278460171184936676238386812500},\label{eq:a}\\
 b&=\decjoin{693614503373734092536437361155200777697197265562}
                  {45843055209221317850552454692640951093751}.\label{eq:b}
\end{align}

Some of the local obstructions can be seen directly. Write $v_\ell(n)$
for the exponent of a prime $\ell$ in $n$. The displayed integers satisfy
\[
 v_5(a)=6,\qquad v_5(b)=0,\qquad
 v_3(a)=v_3(b)=5,\qquad a\equiv0\pmod2,\quad b\equiv1\pmod2.
\]
Each prime multiplication or exact division changes $v_\ell$ by at
most one, so a path with no additive steps needs at least six steps.
If a path of length at most five has exactly one additive step
$x\to x\pm1$, then at most four prime operations occur before and
after that step altogether. Reading from either endpoint shows that
both $x$ and $x\pm1$ are divisible by $3$, which is impossible.
The modular procedure below checks all words, including these two
cases, without requiring a separate classification.

\begin{proposition}\label{prop:distance-six}
For the integers in \eqref{eq:a}--\eqref{eq:b}, one has $d(a,b)=6$.
\end{proposition}

\begin{proof}
For the lower bound, enumerate all words of length at most five in the
four operations $+1,-1,\times p,\div p$, assigning an independent
prime variable to each multiplicative or divisive occurrence. There are
\[
 \sum_{j=0}^{5}4^j=1365
\]
words, including the empty word. Apply Algorithm~\ref{alg:exclude}
with the moduli
\begin{equation}\label{eq:moduli}
 \mathcal{M}=\{\ell,\ell^2:\ell\in\PP,\ 2\leq\ell\leq73\}.
\end{equation}
It excludes all $1365$ words and produces a certificate with $21906$
distinct nodes. The separate checker in Appendix~\ref{app:exclusion}
accepts every node and confirms that every word is covered.
The soundness proofs there imply $d(a,b)\geq6$.

For the reverse inequality, set
\begin{align*}
 p&=\decjoin{452749961069777468738411399464740770981010595670}
                  {9886231801077784553522600406719687081399},\\
 q&=166973,\\
 r&=\decjoin{858529160588943050156391109717017081186495767117}
                  {812089417589062500543606671054056001051192861}.
\end{align*}
These three integers are prime and satisfy the exact identity
\begin{equation}\label{eq:six-identity}
 q(ap+1)-1=(b+1)r.
\end{equation}
Consequently
\begin{equation}\label{eq:six-path}
 a\longrightarrow ap\longrightarrow ap+1\longrightarrow q(ap+1)
 \longrightarrow(b+1)r\longrightarrow b+1\longrightarrow b
\end{equation}
is a path of length six. The identity is checked by integer arithmetic;
the three primality assertions can be checked with a rigorous
primality algorithm, such as \texttt{fmpz\_is\_prime} in
FLINT~\cite{FLINT}. Appendix~\ref{app:paths} gives the corresponding
pseudocode.
\end{proof}

\begin{remark}\label{rem:unbounded}
The lower-bound computation has no bound on the sizes of intermediate
vertices or prime multipliers. It rules out entire operation words by
necessary congruence conditions. A breadth-first search restricted to
a finite interval would not suffice for this purpose.
\end{remark}

\section{The radius}\label{sec:radius}

\begin{proposition}\label{prop:radius}
One has
\[
 4\leq\rad(G)\leq\ecc(1)\leq5.
\]
More precisely, $d(1,n)\leq4$ for every even $n$.
\end{proposition}

\begin{proof}
The upper bounds follow by applying Lemma~\ref{lem:coprime} with
$x=1$. For the lower bound, fix $a\in\N$. We show that
\begin{equation}\label{eq:sparse-three-ball}
 \#\{n\leq X:n\equiv a\pmod2,\ d(a,n)\leq3\}=o_a(X).
\end{equation}
There are only finitely many operation words of length at most three.
We divide them according to the number of additive steps.

If there are no additive steps, cancellation of prime factors expresses
the endpoint as $(a/d)m$, where $d\mid a$ and $\Omega(m)\leq3$.
For each fixed $d$, the number of such endpoints up to $X$ is
\[
 O_a\bigl(X(\log\log X)^2/\log X\bigr),
\]
by the standard estimates
for integers with a bounded number of prime factors;
see~\cite{MontgomeryVaughan}.

If there is exactly one additive step, preserving the parity of the
endpoint requires one of the at most two prime operations to use the
prime $2$. Indeed, multiplication or exact division by an odd prime
preserves parity, whereas the additive step changes it. After fixing
the position and direction of an operation using $2$, at most one prime
variable remains. Its endpoints have the form $up+v$, with fixed
rational $u,v$, or $u/p+v$. For the first form, if $u>0$, the condition
$1\leq up+v\leq X$ restricts $p$ to $O_a(X)$, giving
$O_a(X/\log X)$ endpoints. If $u<0$, there are only finitely many
positive endpoints, and if $u=0$, the expression is constant.
For the second form, write $u=U/D$ and $v=V/D$, where $U,V,D$ are
integers and $D>0$. Integrality of $(U+Vp)/(Dp)$ implies $p\mid U$.
If $U\ne0$, this leaves finitely many choices of $p$; if $U=0$,
the expression is constant.

With two or three additive steps, at most one prime operation remains,
so the same argument applies. This proves~\eqref{eq:sparse-three-ball}.
There are therefore infinitely many vertices outside $B_3(a)$, and
$\ecc(a)\geq4$ for every $a$.
\end{proof}

\section{Consequences of Dickson's conjecture}\label{sec:dickson}

\begin{definition}\label{def:admissible}
Two integer linear forms $L_1(t),L_2(t)$ are \emph{admissible} if
their leading coefficients are positive and, for every prime $\ell$,
there is an integer $t$ such that $\ell\nmid L_1(t)L_2(t)$.
\end{definition}

\begin{conjecture}[Dickson, two linear forms~\cite{Dickson}]\label{conj:dickson}
An admissible pair of distinct, nonproportional integer linear forms
takes simultaneous prime values for infinitely many positive integers.
\end{conjecture}

\begin{lemma}\label{lem:dickson-distance}
Assume Dickson's conjecture. If $x,y$ are coprime and of opposite
parity, then $d(x,y)\leq3$. If they are coprime and odd, then
$d(x,y)\leq4$.
\end{lemma}

\begin{proof}
Put $h=1$ in the first case and $h=2$ in the second. Choose integers
$p_0,r_0$ satisfying $xp_0-yr_0=h$. All integer solutions are
\[
 p=p_0+yt,\qquad r=r_0+xt.
\]
These forms are nonproportional because $xp_0-yr_0=h\ne0$, and
their leading coefficients are positive. If a prime $\ell$ divides
$x$ or $y$, the other coefficient and $h$ are nonzero modulo
$\ell$, so at most one residue of $t$ is forbidden. If
$\ell\nmid xy$ and $\ell>2$, at most two residues are forbidden.
At $2$, either exactly one coefficient is even, or $x,y$ are odd
and $h=2$, in which case the two forms have the same parity and can
both be made odd. Thus the pair is admissible.

Dickson's conjecture supplies prime values of $p$ and $r$. The
equation $xp-yr=h$ gives a path with one multiplication, $h$
subtractions, and one division, of length $h+2$.
\end{proof}

\begin{proposition}\label{prop:dickson}
Assume Dickson's conjecture for two linear forms. Then
\[
 \diam(G)=6,\qquad \rad(G)=\ecc(1)=4.
\]
\end{proposition}

\begin{proof}
Repeat the proof of Proposition~\ref{prop:upper}, using
Lemma~\ref{lem:dickson-distance} instead of Lemma~\ref{lem:coprime}.
This gives distance at most five for endpoints of the same parity and
at most six for arbitrary endpoints. Proposition~\ref{prop:distance-six}
gives equality for the diameter.

Apply Lemma~\ref{lem:dickson-distance} directly to $1,n$. This gives
$d(1,n)\leq4$ for every $n$, and Proposition~\ref{prop:radius}
supplies the matching lower bounds. Together with the unconditional
bounds in Sections~\ref{sec:upper}--\ref{sec:radius}, this completes
the proof of Theorem~\ref{thm:main}.
\end{proof}

\begin{remark}\label{rem:not-equivalent}
The equalities $\diam(G)=6$ and $\rad(G)=4$ have not been shown
here to imply Dickson's conjecture. Proposition~\ref{prop:dickson}
does imply that a pair at distance seven, or a vertex at distance five
from $1$, would contradict the two-form conjecture.
\end{remark}

\section{Arithmetic progressions and local obstructions}\label{sec:progressions}

The modular argument for the lower bound in Section~\ref{sec:lower}
does not extend to an exclusion of all six-step paths based only on
congruence conditions on the endpoints.

\begin{proposition}\label{prop:progressions}
Fix $a\in\N$. Every arithmetic progression contains infinitely
many positive integers $b$ with $d(a,b)\leq6$. More precisely, every
progression containing integers of the same parity as $a$ contains
infinitely many such integers with $d(a,b)\leq5$.
\end{proposition}

\begin{proof}
Replace the modulus $M$ by $\lcm(M,2)$ and, if needed, choose a
parity-compatible subprogression. First prescribe a residue
$\gamma\pmod M$ with $\gamma\equiv a\pmod2$. We construct
primes $p,q,s$ such that
\begin{equation}\label{eq:progression-construction}
 t=s\bigl(q(ap+1)+1\bigr)\equiv\gamma\pmod M.
\end{equation}
Choose an odd prime $p$ such that $c=ap+1$ is a unit modulo every
odd prime power dividing $M$. This is possible by the same reduced
residue argument as in Lemma~\ref{lem:auxiliary}. For an odd prime
power $\ell^e\mid M$, prescribe unit residues $q_\ell,s_\ell$
modulo $\ell^e$ as follows:
\[
 (q_\ell,s_\ell)=
 \begin{cases}
 ((\gamma-1)c^{-1},1),&\ell\mid\gamma,\\
 (c^{-1},\gamma\,2^{-1}),&\ell\nmid\gamma.
 \end{cases}
\]
Here the inverses are taken modulo $\ell^e$. In both cases,
$s_\ell(q_\ell c+1)\equiv\gamma\pmod{\ell^e}$.
At the power of $2$ dividing $M$, use the first formula if $a$ is
even: then $c$ is odd and $\gamma$ is even. If $a$ is odd, then
$c$ is even and $\gamma$ is odd; take $q_2=1$ and
$s_2=\gamma(c+1)^{-1}$.

The Chinese remainder theorem combines these prescriptions into
reduced residue classes for $q$ and $s$. Dirichlet's theorem provides
primes in both classes. Fix $p,q$ and vary $s$ to obtain infinitely
many distinct values of $t$ in~\eqref{eq:progression-construction}.
Each is reached by the five-step path
\[
 a\to ap\to ap+1\to q(ap+1)\to q(ap+1)+1\to t.
\]
This proves the same-parity assertion. For a target residue $\beta$
of the opposite parity, apply the construction to $\gamma=\beta-1$
and take the additional step $t\to t+1$.
\end{proof}

\begin{corollary}\label{cor:congruences}
No nonempty finite system of congruence conditions on $b$, with $a$
fixed, can force $d(a,b)\geq7$ for every positive integer satisfying
those conditions.
\end{corollary}

\begin{proof}
A finite system of congruence conditions describes a union of residue
classes modulo a common modulus. Each nonempty class contains
infinitely many vertices in $B_6(a)$ by
Proposition~\ref{prop:progressions}.
\end{proof}

\section{Open questions}\label{sec:questions}

Theorem~\ref{thm:main} leaves the diameter in $\{6,7\}$ and the
radius in $\{4,5\}$. Dickson's conjecture predicts the smaller value
in each case.

\begin{question}\label{q:diameter}
Is $\diam(G)=6$?
\end{question}

For fixed endpoints, every path of length at most six has one of
$\sum_{j=0}^{6}4^j=5461$ operation words. Some resulting arithmetic
conditions involve three prime variables, for example
\[
 Yr=q(Xp+T)+S,
\]
with fixed integer coefficients. Fixing one prime may lead to two
linear forms that must be prime simultaneously. This is a useful
source of sufficient conditions, but it need not be the only route to
a short path. A proof of diameter six may exploit the choice among
several operation words or the variation of all auxiliary primes at
once. Proposition~\ref{prop:progressions} shows why a finite system
of congruence conditions cannot force a counterexample with one
endpoint fixed.

\begin{question}\label{q:radius}
Is $\rad(G)=4$, and which vertices attain the radius?
\end{question}

The stronger assertion $\ecc(1)=4$ would give $\rad(G)=4$ by
Proposition~\ref{prop:radius}. Since $d(1,n)\leq4$ is already known
for every even $n$, this approach requires only the odd vertices to be
considered. If $\ecc(1)=5$, a different vertex could still have
eccentricity four, so the radius would remain to be determined.

\appendix
\titleformat{\section}{\normalfont\Large\bfseries\filcenter}
  {Appendix \thesection.}{0.6em}{}

\section{A finite modular exclusion algorithm}\label{app:exclusion}

This appendix gives all the data and rules needed to reproduce the
lower bound and to check its certificate separately. The pseudocode
uses finite sets; the accompanying implementation of the forward
procedure stores these sets as bitsets. Let $a,b$ be fixed endpoints
and let $w$ be an operation
word. Each occurrence of multiplication or division receives its own
prime variable. Repeated prime values are allowed. A partial
assignment $\theta$ specifies exact values for some of these
variables; the remaining variables are called \emph{free}. In the
pseudocode, $\theta$ is a vector indexed by word positions: zero marks
a free prime variable, and entries at additive positions are always
zero. A positive entry fixes the corresponding prime exactly.

\subsection{Propagation modulo a prime power}
Fix $m=\ell^e$, and temporarily require every free prime variable
to be different from $\ell$. Its residue is then a unit modulo $m$.
Let $R\subseteq\Z/m\Z$ be the current set of possible residues.
Starting with $R=\{a\bmod m\}$, propagate it through $w$ by
\[
\begin{aligned}
 +1 &: R\longmapsto\{x+1:x\in R\},\\
 -1 &: R\longmapsto\{x-1:x\in R\},\\
 \times p\text{ with }p\text{ fixed}
    &: R\longmapsto\{px:x\in R\},\\
 \div p\text{ with }p\text{ fixed}
    &: R\longmapsto\{y:py\in R\},\\
 \times p\text{ or }\div p\text{ with }p\text{ free}
    &: R\longmapsto\{ux:x\in R,\ u\in(\Z/m\Z)^\times\}.
\end{aligned}
\]
All operations in these sets are modulo $m$. The last rule is valid
for division because inversion permutes the units. For computation,
one may use
\begin{equation}\label{eq:unit-orbits}
 \{ux:u\in(\Z/m\Z)^\times\}
 =\{y:\gcd(y,m)=\gcd(x,m)\},
\end{equation}
where residues are represented by integers from $0$ to $m-1$.

The following set-based pseudocode implements these rules. The
preimage in the division case also applies when $p$ divides $m$.
\par\medskip
\begingroup
\small
\noindent\begin{minipage}{\linewidth}
\begin{verbatim}
Possible(w, theta, m):
    R := {a mod m}
    for j = 1, ..., length(w):
        if w[j] is +1:
            R := {(x+1) mod m : x in R}
        else if w[j] is -1:
            R := {(x-1) mod m : x in R}
        else if theta[j] = 0:
            D := {gcd(x,m) : x in R}
            R := {y in 0,...,m-1 : gcd(y,m) in D}
        else:
            p := theta[j]
            if w[j] is multiply:
                R := {p*x mod m : x in R}
            else:
                R := {y in 0,...,m-1 : (p*y mod m) in R}
        if R is empty:
            return FALSE
    return (b mod m) in R
\end{verbatim}
\end{minipage}\par
\endgroup
\medskip

Denote by $\operatorname{Possible}(w,\theta,m)$ the assertion that
$b\bmod m$ belongs to the final residue set. These rules may admit
modular paths that do not lift to positive integral paths. This causes
no problem: the sets are necessary-condition overestimates, and only
their failure to contain the target is used.

\subsection{Exceptional prime values and recursion}
If $\operatorname{Possible}(w,\theta,\ell^e)$ is false, a genuine
path realizing $w$ must assign the exact prime $\ell$ to at least
one free variable. This gives a finite branching rule.

\begin{algorithm}[Exclusion and certificate generation]\label{alg:exclude}
The procedure $\operatorname{Exclude}(w,\theta)$ returns true only
when the word $w$ is impossible under the partial assignment
$\theta$. For each successful call, a table $C$ records a modulus
whose failure forces the recursive branches. A separate table $H$
memoizes return values. Both tables are initially empty for each pair
of endpoints.
\end{algorithm}

\par\medskip
\begingroup
\small
\noindent\begin{minipage}{\linewidth}
\begin{verbatim}
Exclude(word w, partial assignment theta):
    if (w,theta) is in H:
        return H[w,theta]
    J := the free prime-variable positions in w
    for (ell, m) in the ordered modulus list:
        if not Possible(w, theta, m):
            all_children_excluded := TRUE
            for j in J, in increasing order:
                theta_j := theta with the j-th prime set to ell
                if not Exclude(w, theta_j):
                    all_children_excluded := FALSE
                    break
            if all_children_excluded:
                C[w,theta] := (ell,m)
                H[w,theta] := TRUE
                return TRUE
    H[w,theta] := FALSE
    return FALSE
\end{verbatim}
\end{minipage}\par
\endgroup
\medskip

The driver enumerates the words and retains the records needed to
check the excluded roots.
\par\medskip
\begingroup
\small
\noindent\begin{minipage}{\linewidth}
\begin{verbatim}
ExcludeAll(a, b, L):
    initialize H and C as empty tables
    unresolved := empty list
    roots := empty list
    for k = 0, ..., L:
        for w in lexicographic order (+1,-1,multiply,divide)^k:
            theta := zero_vector(k)
            if Exclude(w, theta):
                append (w,theta) to roots
            else:
                append w to unresolved
    retain in C only nodes reachable from roots
    attach the input header (a,b,L) to C
    return (unresolved, C)
\end{verbatim}
\end{minipage}\par
\endgroup
\medskip

The children of a node $(w,\theta)$ with record $(\ell,m)$ are
all nodes obtained by setting one free position to $\ell$.
Reachability in the last procedure is computed using these children.
When no variables are free, the empty family of children is accepted.
Returning false means only that these modular tests did not exclude
the word. It does not assert the existence of a path. Memoization and
the bitset representation affect efficiency, not the logic.

\begin{proposition}\label{prop:algorithm-sound}
The procedure terminates. If $\operatorname{Exclude}(w,\theta)$
returns true, no positive-integer path from $a$ to $b$ realizes $w$
with prime values extending $\theta$.
\end{proposition}

\begin{proof}
Every recursive call assigns one previously free variable, so the
depth is bounded by the number of prime operations in $w$.
Prove soundness by induction on that number of free variables.
If there are none, a failed modular propagation contradicts a
necessary congruence condition for any genuine path.

Otherwise, suppose the procedure returns true at modulus $\ell^e$.
If a genuine path assigned no free variable the prime $\ell$, all
free residues would be units, and its reduction would be admitted by
the propagation rules. This contradicts the failed propagation. Thus
at least one free variable equals $\ell$. The corresponding recursive
branch excludes the path by the induction hypothesis. Since every
such branch returns true, no genuine path remains.
\end{proof}

\subsection{A separate certificate checker}
The certificate is a finite table keyed by $(w,\theta)$, with a pair
$(\ell,m)$ at each key. To check a node, propagate possible residues
backward from $b$. The checker below uses explicit sets and forms unit
orbits by multiplication, independently of the bitsets and gcd-class
calculation used by the forward implementation.

\par\medskip
\begingroup
\small
\noindent\begin{minipage}{\linewidth}
\begin{verbatim}
BackwardPossible(w, theta, m):
    S := {b mod m}
    U := {u in 0,...,m-1 : gcd(u,m) = 1}
    for j = length(w), ..., 1:
        if w[j] is +1:
            S := {(y-1) mod m : y in S}
        else if w[j] is -1:
            S := {(y+1) mod m : y in S}
        else if theta[j] = 0:
            S := {u*y mod m : u in U, y in S}
        else:
            p := theta[j]
            if w[j] is multiply:
                S := {x in 0,...,m-1 : (p*x mod m) in S}
            else:
                S := {p*y mod m : y in S}
    return (a mod m) in S
\end{verbatim}
\end{minipage}\par
\endgroup
\medskip

For free division or multiplication, inversion permutes the units, so
the same rule is valid in either direction. For fixed division, the
backward image is multiplication by $p$, even when $p$ is not a unit.

\begin{algorithm}[Certificate checker]\label{alg:check-certificate}
Given endpoints $a,b$, a nonnegative length bound $L$, and a
certificate $C$, apply the following checks. A failed requirement
rejects the certificate.
\end{algorithm}

\par\medskip
\begingroup
\small
\noindent\begin{minipage}{\linewidth}
\begin{verbatim}
CheckCertificate(a, b, L, C):
    require the input header of C to equal (a,b,L)
    require all keys (w,theta) to be distinct
    for every record C[w,theta] = (ell,m):
        require w to be a word in the four allowed operations
        require theta to have length length(w)
        require additive positions of theta to be zero
        require every other entry of theta to be zero or prime
        require ell to be prime and m in {ell,ell^2}
        require not BackwardPossible(w,theta,m)
        for each free prime-variable position j:
            theta_j := theta with the j-th prime set to ell
            require (w,theta_j) to be a key in C
    for k = 0, ..., L:
        for every word w of length k:
            require (w,zero_vector(k)) to be a key in C
    return ACCEPT
\end{verbatim}
\end{minipage}\par
\endgroup
\medskip

\begin{proposition}\label{prop:certificate-sound}
If Algorithm~\ref{alg:check-certificate} accepts, then $d(a,b)>L$.
\end{proposition}

\begin{proof}
Backward propagation tests the same modular relations as forward
propagation, with their directions reversed. Thus a failed backward
test at $\ell^e$ implies that any genuine path realizing the recorded
word and assignment uses $\ell$ at a free position. Every resulting
child is present in the certificate and has one fewer free variable.
Induction on the number of free variables therefore excludes every
recorded node, starting with those having no free variables. All
unassigned words of length at most $L$ are present, so every path of
such a length is excluded.
\end{proof}

\subsection{Input and output for the distance-six example}
Use the endpoints~\eqref{eq:a}--\eqref{eq:b} and $L=5$.
Take the primes
\[
 2,3,5,7,11,13,17,19,23,29,31,37,41,43,47,53,59,61,67,71,73.
\]
The ordered modulus list consists first of these primes and then their
squares, in the same order. The result is
\begin{center}
\begin{tabular}{@{}ccc@{}}
\toprule
Maximum word length & Words examined & Unresolved words \\
\midrule
0 & 1    & 0 \\
1 & 5    & 0 \\
2 & 21   & 0 \\
3 & 85   & 0 \\
4 & 341  & 0 \\
5 & 1365 & 0 \\
\bottomrule
\end{tabular}
\end{center}
All arithmetic is exact integer or finite-set arithmetic. In
particular, no numerical tolerance and no search bound on prime
values enters this calculation.

With the word order and memoization specified above, the forward
implementation visits $21906$ distinct exclusion states and performs
$84421$ calls to \texttt{Possible}. The exported certificate contains
$21906$ nodes. The separate checker accepts all of them and verifies
the presence of all $1365$ unassigned words. These results were
reproduced with Python~3.12.

The accompanying ancillary directory \path{anc/} contains
\path{verify_draft.py}, which reads the displayed integers directly
from the source and generates the certificate, and
\path{check_certificate.py}, which checks the certificate against the
same input endpoints and length bound. The files
\path{modular_certificate.json}, \path{verification_results.json},
and \path{certificate_check_results.json} record the certificate and
the run results. Download and extract the complete source archive so
that the manuscript source is in the parent directory of \path{anc/}.
From \path{anc/}, the complete verification is reproduced by
\par\medskip
\begingroup
\small
\noindent\begin{minipage}{\linewidth}
\begin{verbatim}
python verify_draft.py ../additive_multiplicative_graph.tex
python check_certificate.py ../additive_multiplicative_graph.tex
\end{verbatim}
\end{minipage}\par
\endgroup
\medskip
The modular procedures use the Python standard library. The first
command also performs the primality proofs in Appendix~\ref{app:paths}
and requires \texttt{python-flint}; the second command has no external
library dependency.

\section{Checking the six-step path}\label{app:paths}

For the six-step example, the independent upper-bound check is:
\par\medskip
\begingroup
\small
\noindent\begin{minipage}{\linewidth}
\begin{verbatim}
Read the exact integers a, b, p, q, r from Section 3.
Assert IsPrime(p), IsPrime(q), and IsPrime(r).
Assert q*(a*p + 1) - 1 == (b + 1)*r.
Return [a, a*p, a*p+1, q*(a*p+1), (b+1)*r, b+1, b].
\end{verbatim}
\end{minipage}\par
\endgroup
\medskip
Here \texttt{IsPrime} must be a primality-proving routine, rather
than a probable-prime test. For example, FLINT's
\texttt{fmpz\_is\_prime} returns $1$ only after proving
primality~\cite{FLINT}. Combined with Appendix~\ref{app:exclusion},
this checks both inequalities needed for $d(a,b)=6$.
In the verification run, \texttt{python-flint}~0.9.0 with
FLINT~3.6.0 proved all three integers prime, and exact integer
arithmetic verified the identity~\eqref{eq:six-identity}.

\end{document}